\documentclass[11pt]{amsart}
\usepackage{amssymb,amsmath,txfonts}
\newtheorem{theorem}{Theorem}[section]
\newtheorem{prop}[theorem]{Proposition}
\newtheorem{lemma}[theorem]{Lemma}

\newtheorem{corollary}[theorem]{Corollary}

\numberwithin{equation}{section}

\begin{document}

\title{Note on the backwards uniqueness of mean curvature flow}
\author{Zhuhong Zhang}
\address{School of Mathematical Sciences, South China Normal Univeristy, Guangzhou, P. R. China 510275}
\email{juhoncheung@sina.com}


\begin{abstract}
In this note, we will show a backwards uniqueness theorem of the mean curvature flow with bounded second fundamental form in arbitrary codimension.
\end{abstract}

\maketitle
\section{Introduction}

Let $X_0:\ M^m\rightarrow R^{m+n}$ be an $m$-dimensional immersion in the $(m+n)-$dimensional Euclidean space. The mean curvature flow (MCF) is a deformation of the position vector $X$, starting from $X_0$ at $t=0$, in the direction of the mean curvature vector $H$,
$$ X : M^m\times [0, T) \rightarrow R^{m+n}, \qquad \frac{\partial}{\partial t}X=H.$$

In general, the MCF can be written as a quasilinear parabolic equation
$$\frac{\partial}{\partial t}X=\Delta_{g(t)}X,$$
where $g(t)$ is the induced metric on  $M^m$, and $\Delta_{g(t)} X$  is the harmonic map Laplacian from $(M, g(t))$ to $R^{m+n}$. When $M^m$ is a compact hypersurface  in the Eucidean space $R^{m+1}$, it is well-known that 
the MCF has a unique short time solution. For $m$-dimensional complete immersed local Lipschitz hypersurface in $R^{m+1}$, we refer the readers to \cite{EH} for more information.
While the short time existence and uniqueness of MCF have not been established  in the literature for complete immersed submanifolds of arbitrary codimensions in a general ambient Riemannian manifold, 
Chen and Yin \cite{CY} have proved the uniqueness of the MCF of general codimensions and in general ambient manifolds with bounded geometry.

Our main theorem of this paper is the following backwards uniqueness result of MCF.

\begin{theorem}
Let $M^m$ be an $m-$dimensional smooth manifold. Suppose $X,\tilde{X}: M^m \rightarrow R^{m+n}$ are two smooth complete solutions of the mean curvature flow with bounded second fundamental form:  $|h_{ij}|_{g}\leq K$ and $|\tilde{h}_{ij}|_{\tilde{g}}\leq \tilde{K}$ on $M^m\times [0,T]$. If $X(x,T)=\tilde{X}(x,T)$ for all $x\in M^m$, then $X(x,t)=\tilde{X}(x,t)$ for all $x \in M^m$ and all $t\in [0,T]$.
\end{theorem}

As a direct consequence, we have

\begin{corollary}
If $X(t): M^m \rightarrow R^{m+n}$ is a solution to  the MCF with bounded second fundamental form on $[0, T]$, and $g(t)$ is the induced metric.  Let $\bar{\sigma}$ be an isometry of $R^{m+n}$ such that there is an isometry $\sigma$ of $(M^m, g(T))$ satisfying
$$(\bar{\sigma}\circ X(T))(x)=(X(T)\circ\sigma)(x)$$
for all $x\in M$. Then we have
$$(\bar{\sigma}\circ X(t))(x)=(X(t)\circ\sigma)(x)$$
for all $(x, t)\in M\times [0, T]$. In particular, the isometry subgroup of $(M^m, g(T))$ induced by an isometry subgroup of $R^{m+n}$ remains to be an isometry subgroup of $(M^m, g(t))$ for any $t\in [0, T]$. 
\end{corollary}

The study of backwards-uniqueness and unique-continuation properties for solutions of parabolic equations has a long history ( c.f. Mizohata\cite{M}, Yamabe\cite{Y}, Lees-Protter\cite{LP}, Agmon-Nirenberg\cite{AN} , Lin\cite{L} and the references therein). In a recent work of Kotschwar \cite{K10}, he showed a backwards uniqueness theorem for the Ricci flow, which is a weakly parabolic system, with bounded curvature. In fact, Kotschwar \cite{K10} was able to deal with the backwards uniqueness problem for some more general weakly parabolic system. One of the key points in his work is to reduce the backwards uniqueness problem for weakly parabolic system to one for a larger system of coupled differential inequalities, a pair of partial differential inequality and ordinary differential inequality.

On the other hand, Wang in \cite{LW14} obtained an uniqueness result for asymptotically conical  co-dim one MCF self-shrinker,  by reducing to the backwards uniqueness of related self-shrinking solutions to the MCF. By the same general strategy, Kotschwar and Wang \cite{KW} proved a similar uniqueness theorem for asymptotically conical  shrinking gradient Ricci solitons.

In this paper, motivated by geometric applications of submanifold theory, we consider the backwards uniqueness problem of the MCF. 
As in \cite{K10}, we want to convert Theorem  1.1 into a backwards uniqueness for a PDE-ODE system. But there is a difficulty that we can not construct geometric tensor fields as in \cite{K10}, since the second fundamental form $h_{ij}$ and $\tilde{h}_{ij}$ are defined in different domains $X$ and $\tilde{X}$ on the target space.
To deal with this problem, we notice that the target space is the Euclidean space, which has a fixed globe coordinate system. So we can decompose all the geometric quantities on $X$ and $\tilde{X}$, and get some abstract intrinsic tensors on $M$ from the pullbacks. Fortunately, by a long computations, we can show that these tensors satisfy a PDE-ODE system, from which the  backwards uniqueness result follows.

{\bf Acknowledgements.} The author was partially supported by NSFC 11301191.
This first draft was written while the author was visiting the University of Macau from February 
to August 2014, where he was supported by Science and Technology Development Fund (Macao S.A.R.) 
FDCT/016/2013/A1. He is grateful to Professor Huai-Dong Cao for encouragement and very helpful discussions.

\section{Preliminaries}

Let $X(x,t):M^m\rightarrow R^{m+n}$
be one-parameter family of smooth immersions in $R^{m+n}$ evolves by the mean curvature flow
$$\frac{\partial}{\partial t}X(x,t)=H(x,t).$$

For any fixed $X(t)$, we have an induced metric $g$ on $M^m$. And the second fundamental form $II$ define by
$II(V,W)=\bar{\nabla}_{\tilde{V}}\tilde{W}-\nabla_{\tilde{V}}\tilde{W}=(\bar{\nabla}_{\tilde{V}}\tilde{W})^{\bot}$, where $\bar{\nabla}$ and $\nabla$ are the covariant derivatives of $R^{m+n}$ and $(M^m, g)$ respectively, $\tilde{V}$ and $\tilde{W}$ are any smooth extensions of $V$ and $W$ on $R^{m+n}$.

For a local coordinate system $\{x^i\}$ of $M$, the metric and second fundamental form on $X$ can be computed as follows
$$g_{ij}=\bigg(\frac{\partial X}{\partial x^i}, \frac{\partial X}{\partial x^j}\bigg),$$
$$h_{ij}=II\bigg(\frac{\partial}{\partial x^i}, \frac{\partial}{\partial x^j}\bigg)=\nabla_i X_j =\frac{\partial^2 X}{\partial x^i \partial x^j}-\Gamma^k_{ij}X_k,$$
and the mean curvature vector $H=g^{ij}h_{ij}$.

In the following, we always give the standard coordinate system $\{y^\alpha\}$ for $R^{m+n}$ and then the immersion can be denote by $X(x,t)=(X^1,\cdots, X^{\alpha}, \cdots, X^{m+n})$. Thus the metric and second fundamental form are 
given by the following
$$g_{ij}=\bigg(\frac{\partial X}{\partial x^i}, \frac{\partial X}{\partial x^j}\bigg)=\sum\limits_\alpha X^\alpha _iX^\alpha_j,$$
$$h_{ij}=h^\alpha_{ij}\frac{\partial}{\partial y^{\alpha}}=\bigg(\frac{\partial^2 X^\alpha}{\partial x^i \partial x^j}-\Gamma^k_{ij}X^\alpha_k\bigg)\frac{\partial}{\partial y^{\alpha}}.$$

The immersion $X(t)$, second fundamental form $h_{ij}(t)$ and mean curvature vector $H(t)$ are all extrinsic geometric quantities. However, for the fixed coordinate $\{y^\alpha\}$ of $R^{m+n}$, we have some simply facts.

\begin{lemma}
For a family of immersions $X(t)$, we have the following facts:\\
 1)\quad  $X^{\alpha}(t)$ is a family of smooth functions on M for any $\alpha=1, \cdots, m+n$; \\
 2)\quad $g_{ij}(t)$ is a family of metrics on M; \\
 3)\quad  $h^{\alpha}_{ij}(t)$ is a family of tensors on M for any $\alpha=1, \cdots, m+n$; \\
 4)\quad  $H^{\alpha}(t)$ is a family of smooth functions on M for any $\alpha=1, \cdots, m+n$. \\
\end{lemma}

Thus, we have a Riemannian metrics $g_{ij}(t)$ on $M$, and get some intrinsic geometric quantities $X^\alpha(t)$, $h_{ij}^\alpha(t)$ and $H^\alpha(t)$. Further more, the above geometric quantities satisfies the 
following evolution equations.

\begin{prop}

\begin{align*}
 1)\qquad & \frac{\partial}{\partial t} X^\alpha_i=\nabla_i H^\alpha , \\
 2)\qquad & \frac{\partial}{\partial t} g_{ij}=-2\sum\limits_\alpha H^\alpha h^\alpha_{ij}, \\
 3)\qquad & \frac{\partial}{\partial t} \Gamma_{ij}^k=-g^{kl}\Big[\nabla_i\Big(\sum\limits_\alpha H^\alpha h^\alpha_{jl}\Big)+\nabla_j\Big(\sum\limits_\alpha H^\alpha h^\alpha_{il}\Big)-\nabla_l\Big(\sum\limits_\alpha H^\alpha h^\alpha_{ij}\Big)\Big],\\
 4)\qquad & \frac{\partial}{\partial t} h^\alpha_{ij}=\nabla_i \nabla_j H^\alpha- \frac{\partial}{\partial t} \Gamma_{ij}^k \cdot X^\alpha_k.
\end{align*}

\end{prop}

For the further computation, we need the Simon's identity
\begin{lemma}
\begin{align*}
\nabla_i \nabla_j H^\alpha =& \Delta h^\alpha_{ij} - g^{pq}\Big(\nabla_iR_{jp}+\nabla_jR_{ip}-\nabla_pR_{ij}\Big)X^\alpha_q \\
                                            &  \qquad +2g^{kp}g^{lq}R_{ikjl}h^\alpha_{pq} - g^{pq}R_{ip}h^\alpha_{jq} -g^{pq}R_{jp}h^\alpha_{iq} .
\end{align*}
where the curvature tensor $R_{ijkl}$  can be obtain from the Gauss equation $$R_{ijkl}=\sum\limits_\alpha \Big(h^\alpha_{ik}h^\alpha_{jl} -h^\alpha_{il}h^\alpha_{jk}\Big),$$
and the Ricci curvature tensor
$R_{ij}=g^{kl}R_{ikjl}=\sum\limits_\alpha \Big(H^\alpha h^\alpha_{ij} - g^{kl}h^\alpha_{ik}h^\alpha_{jl}\Big) $.

\end{lemma}

\section{The PDE-ODE system}

Let $X(t)$ and $\tilde{X}(t)$ be two complete solutions of mean curvature flow on $M$ for $t\in[0,T]$, 
which induce two family of functions $X^\alpha$ and $\tilde{X}^\alpha$, two Riemannian metrics $g(t)$ and $\tilde{g}(t)$. Hence by Lemma 2.1, we have two family of tensors $h^\alpha$ and $\tilde{h}^\alpha$.  
Consequnesly, we denote by $\nabla$ and $\tilde{\nabla}$ their Levi-Civita connections, and by $R$ and $\tilde{R}$ their Riemannian curvature tensors. 

Fixed the metric $g(t)$ on $M$. We introduce the tensor field  $d\doteqdot g-\tilde{g}$, $N\doteqdot\nabla-\tilde{\nabla}$, $W\doteqdot\nabla N$,.

On the other hand, for any fixed $\alpha$, we denote by $w^\alpha \doteqdot \nabla X^\alpha-\tilde{\nabla}\tilde{X}^\alpha$, $U^\alpha\doteqdot h^\alpha-\tilde{h}^\alpha$, and $V^\alpha\doteqdot\nabla h^\alpha-\tilde{\nabla}\tilde{h}^\alpha$. Let $Y(t)=U^1(t)\oplus\cdots\oplus U^{m+n}(t)\oplus  V^1(t)\oplus\cdots\oplus V^{m+n}(t) $ and $Z(t)=w^1(t)\oplus \cdots \oplus w^{m+n}(t) \oplus d(t)\oplus N(t) \oplus W(t)$. Then we have the following PDE-ODE system.

\begin{theorem}
Under the hypothesis of Theorem 1.1. Then for any $0<\delta<T$, there exist a constant $C=C(\delta, T, K, \tilde{K})>0$ such that

\begin{align}
 \bigg|\bigg(\frac{\partial}{\partial t}-\Delta_{g(t)}\bigg)Y\bigg|^2_{g(t)} & \leq C\big(|Y|^2_{g(t)}+|\nabla Y|^2_{g(t)}+|Z|^2_{g(t)}\big),  \\
 \bigg|\frac{\partial}{\partial t}Z\bigg|^2_{g(t)} & \leq C\big(|Y|^2_{g(t)}+|\nabla Y|^2_{g(t)}+|Z|^2_{g(t)}\big)
\end{align}
on $M\times [\delta,T]$.
\end{theorem}

Let us denote by  $A\ast B$ any tensor product of two tensors $A$ and $B$ by $g$ when we do not need the precise expression.
For example, since $g^{ij}-\tilde{g}^{ij}=-\tilde{g}^{ia} g^{ja} d_{ab}$, and
\begin{align*}
\frac{\partial}{\partial t} d_{ij} &=-2\sum\limits_\alpha  \bigg(  g^{kl}h_{kl}^{\alpha}h_{ij}^{\alpha} -  \tilde{g}^{kl} \tilde{h}_{kl}^{\alpha} \tilde{h}_{ij}^{\alpha} \bigg) \\
  &=-2\sum\limits_\alpha \bigg(   g^{kl}h_{kl}^{\alpha} (h_{ij}^{\alpha} -\tilde{h}_{ij}^{\alpha} )+ g^{kl}(h_{kl}^{\alpha}-\tilde{h}_{kl}^{\alpha})\tilde{h}_{ij}^{\alpha}  + (g^{kl}-\tilde{g}^{kl}) \tilde{h}_{kl}^{\alpha} \tilde{h}_{ij}^{\alpha}    \bigg), \\
\end{align*}
so $g^{-1}-\tilde{g}^{-1}= \tilde{g}^{-1}\ast d$, and 
$$\frac{\partial}{\partial t} d=\sum\limits_\alpha \bigg( h^\alpha \ast U^\alpha +  \tilde{h}^\alpha\ast U^\alpha+\tilde{g}^{-1}\ast d\ast \tilde{h}^\alpha \ast \tilde{h}^\alpha \bigg).$$

Furthermore, we denote by  eeeeee3ee a linear combination of tensors $A \ast C$, $B \ast B \ast C$, etc. Hence
$$g^{-1}-\tilde{g}^{-1}=\mathcal{L}(d; \quad \tilde{g}^{-1}), \qquad \frac{\partial}{\partial t} d=\mathcal{L}(d, U^{\alpha};\quad \tilde{g}^{-1}, h^{\alpha}, \tilde{h}^{\alpha}).$$

Now we will give the evolutions of the component of $Y$ and $Z$.
\begin{lemma}
If $g$, $\tilde{g}$, $d$, $N$, $W$, $w^\alpha$, $U^\alpha$, and $V^\alpha$ as above, then
\begin{align*}
 1)\quad & \frac{\partial}{\partial t} w^\alpha=\mathcal{L}(d, V^\alpha ; \quad  \tilde{g}^{-1}, \tilde{\nabla}\tilde{h}^\alpha) , \\
 2)\quad & \frac{\partial}{\partial t} d=\mathcal{L}(d, U^{\alpha}; \quad \tilde{g}^{-1}, h^{\alpha}, \tilde{h}^{\alpha}), \\
 3)\quad & \frac{\partial}{\partial t} N= \mathcal{L}(d, U^{\alpha}, V^\alpha; \quad \tilde{g}^{-1}, h^{\alpha}, \tilde{h}^{\alpha}, \tilde{\nabla} \tilde{h}^ \alpha ),\\
 4)\quad & \frac{\partial}{\partial t} W= \mathcal{L}(d, N, U^{\alpha}, V^\alpha;      \quad
                        \tilde{g}^{-1}, h^{\alpha}, \nabla h^\alpha, \nabla^2 h^\alpha,  \tilde{h}^{\alpha}, \tilde{\nabla} \tilde{h}^ \alpha, \tilde{\nabla} ^2 \tilde{h}^ \alpha),\\
 5)\quad & \bigg(\frac{\partial}{\partial t} - \Delta \bigg) U^\alpha= \mathcal{L}(d, w^\alpha, N, W, U^{\alpha}, V^\alpha;      \quad
                  \tilde{g}^{-1}, \tilde{\nabla}\tilde{X}^\alpha, h^{\alpha}, \nabla h^\alpha,  \tilde{h}^{\alpha}, \tilde{\nabla} \tilde{h}^ \alpha, \tilde{\nabla} ^2 \tilde{h}^ \alpha),\\
 6)\quad & \bigg(\frac{\partial}{\partial t} - \Delta \bigg)  V^\alpha= \mathcal{L}(d, w^\alpha, N, W, U^{\alpha}, V^\alpha;      \quad
                  \tilde{g}^{-1}, \tilde{\nabla}\tilde{X}^\alpha, h^{\alpha}, \nabla h^\alpha, \nabla^2 h^\alpha, \tilde{h}^{\alpha}, \tilde{\nabla} \tilde{h}^ \alpha, \tilde{\nabla} ^2 \tilde{h}^ \alpha, \tilde{\nabla} ^3 \tilde{h}^ \alpha).\\
 \end{align*}
\end{lemma}

\begin{proof}
Since
$$g^{-1}-\tilde{g}^{-1}=\tilde{g}^{-1}\ast d,$$
then by Proposition 2.2, we can get 1),  2) and 3).

Note that $\frac{\partial}{\partial t}\nabla N=\nabla \frac{\partial}{\partial t}N+\frac{\partial}{\partial t}\Gamma\ast N$ and $\frac{\partial}{\partial t}\Gamma= h^\alpha\ast \nabla h^\alpha$.
Then by the fact that $\tilde{\nabla} A - \nabla A = N \ast A$ for any tensor $A$, we obtain 4).

For 5) and 6),  from Proposition 2.2, 
$$\bigg(\frac{\partial}{\partial t} - \Delta \bigg) h^\alpha=\sum\limits_\beta \bigg( h^\beta\ast \nabla h^\beta\ast \nabla X^\alpha + h^\beta\ast h^\beta\ast h^\alpha \bigg).$$

Now
$$\bigg(\frac{\partial}{\partial t} - \Delta \bigg) U^\alpha=\bigg(\frac{\partial}{\partial t} - \Delta \bigg) h^\alpha -\bigg(\frac{\partial}{\partial t} - \tilde{\Delta} \bigg) \tilde{h}^\alpha - \bigg(\tilde{\Delta}- \Delta \bigg) \tilde{h}^\alpha,$$
and since $\tilde{\nabla}\tilde{\nabla}A=\nabla \nabla A + N \ast \tilde{\nabla} A + W \ast A  +  N \ast N \ast A $, we have
$$\tilde{\Delta}\tilde{h}^\alpha=\Delta \tilde{h}^\alpha+  N \ast \tilde{\nabla} \tilde{h}^\alpha + W \ast \tilde{h}^\alpha +  N \ast N \ast \tilde{h}^\alpha + \tilde{g}^{-1}\ast d\ast\tilde{\nabla}\tilde{\nabla}\tilde{h}^\alpha,$$
hence we can get 5).

On the other hand, 
$$\bigg(\frac{\partial}{\partial t} - \Delta \bigg) V^\alpha=\bigg(\frac{\partial}{\partial t} - \Delta \bigg) \nabla h^\alpha -\bigg(\frac{\partial}{\partial t} - \tilde{\Delta} \bigg) \tilde{\nabla}\tilde{h}^\alpha - \bigg(\tilde{\Delta}- \Delta \bigg)\tilde{\nabla}\tilde{h}^\alpha.$$
Note that by the Gauss equation, we obtain
\begin{align*}
 \frac{\partial}{\partial t} \nabla h^\alpha &=\nabla  \frac{\partial}{\partial t}h^\alpha + \frac{\partial}{\partial t}\Gamma \ast h^\alpha \\
                &=\nabla (\Delta h^\alpha + h^\beta\ast \nabla h^\beta\ast \nabla X^\alpha + h^\beta\ast h^\beta\ast h^\alpha) + h^\beta \ast \nabla{h}^\beta \ast h^\alpha \\
                &=\Delta \nabla h^\alpha  + \nabla h^\beta\ast \nabla h^\beta\ast \nabla X^\alpha + h^\beta\ast \nabla \nabla h^\beta\ast \nabla X^\alpha \\
                &\qquad  + h^\beta \ast \nabla{h}^\beta \ast h^\alpha+ h^\beta\ast h^\beta\ast \nabla h^\alpha.
\end{align*} 

On the other hand, since $\tilde{\nabla}\tilde{\nabla}A=\nabla \nabla A + N \ast \tilde{\nabla} A + W \ast A  +  N \ast N \ast A $, we have
 $$\tilde{\Delta}\tilde{\nabla}\tilde{h}^\alpha=\Delta \tilde{\nabla}\tilde{h}^\alpha 
+ N \ast \tilde{\nabla} \tilde{\nabla}\tilde{h}^\alpha + W \ast \tilde{\nabla}\tilde{h}^\alpha +  N \ast N \ast  \tilde{\nabla}\tilde{h}^\alpha+ \tilde{g}^{-1}\ast d\ast\tilde{\nabla}^{3}\tilde{h}^\alpha.$$

Combine the above equations, we can get 6).
\end{proof}

Now follow a similar argument of Kotschwar \cite{K10}, we can prove Theorem 3.1.
\begin{proof}
By the stander Berstein estimate of MCF, once $h$ and $\tilde{h}$ are uniformly bounded, all higher derivatives $\nabla^{(k)} h$ and $\tilde{\nabla}^{(k)} \tilde{h}$ are uniformly bounded. 
This can be shown by induction and originally prove for hypersurfaces by Huisken \cite{Hui84}. Thus for any $0<\delta<T$ and $k\ge 0$, there exist constant $C_k$ and $\tilde{C}_k$, such that
$$|\nabla^{(k)} h|_{g(t)}\le C_k, \qquad and\qquad |\tilde{\nabla}^{(k)} \tilde{h}|_{\tilde{g}(t)}\le \tilde{C}_k$$
on $t\in[\delta, T]$. On the other hand, by the Gauss equation, the curvature $R$ and $\tilde{R}$ are all bounded, which implies the uniformly equivalent of $g$ and $\tilde{g}$, that is, 
there exist constants $\gamma$, such that
$$\gamma^{-1}g\le\tilde{g}\le\gamma g$$
on $t\in [0, T]$. So $ |\tilde{\nabla}^{(k)} \tilde{h}|_{g(t)}\le \tilde{C}_k$ for some new $\tilde{C}_k$.  

Since $$\sum_\alpha|\nabla X^\alpha|^2_g(t)=\sum_\alpha g^{ij}X_i^\alpha X_j^\alpha=g^{ij}\sum_\alpha X_i^\alpha X_j^\alpha=g^{ij}g_{ij}=m,$$
$\nabla X^\alpha$, $\tilde{\nabla} \tilde{X}^\alpha$ and hence $w^\alpha$ are bounded.

It is not hard to see that $d$, $U^\alpha$ and $V^\alpha$ are all bounded on $[\delta, T]$. We remind to show $N$ and $W$ are bounded.

Since $N(T)=0$, by the evolution of $N$, we have
$$|N|_{g(t)} = \bigg|N_{ij}^k(x,t)-N_{ij}^k(x, T)\bigg|_{g(t)}  \le \int_t^T\bigg|\frac{\partial}{\partial s} N_{ij}^k(x,s)\bigg|_{g(t)}ds\le C'T.$$
Similarly, we can obtain the bounded of $W$.

Then the theorem follow the Cauchy-Schwarz inequality.
\end{proof}

\section{Backwards uniqueness of the mean curvature flow}

Now we can prove our main theorem.
\begin{proof}\emph{(of Theorem 1.1)}
Since $h_{ij}$ and $\tilde{h}_{ij}$ are bounded, the curvature $R$, tensor fields $Y$, $Z$ and $\nabla Y$ are all bounded. By using the backwards uniqueness theorem of Kotschwar \cite{K10}, we obtain $Y=0$, $Z=0$ on $M\times[0, T]$.
In particular, $H^{\alpha}=0$ for $\alpha=1, \cdots, m+n$, which implies $H(x, t)=\tilde{H}(x, t)$ for all $(x, t)\in M\times [0, T]$. So
$$\frac{\partial}{\partial t}(X-\tilde{X})=H-\tilde{H}=0,\quad and\quad X(T)-\tilde{X}(T)=0,$$
and we complete the proof of Theorem 1.
\end{proof}

Corollary 1.2 is a direct consequence of Theorem 1.1.
\begin{proof}\emph{(of Corollary 1.2)}
Indeed, let $\bar{\sigma}$ and $\sigma$ be two isometries of $R^n$ and $(M, g(T))$ respectively such that $(\bar{\sigma}\circ X(T))(x)=(X(T)\circ\sigma)(x)$. 
Since $(\bar{\sigma}\circ X(t))(x)$ and $(X(t)\circ\sigma)(x)$ are two solutions of MCF
with bounded second fundamental form and $(\bar{\sigma}\circ X(T))(x)=(X(T)\circ\sigma)(x)$. Then by Theorem 1.1, we have
$$(\bar{\sigma}\circ X(t))(x)=(X(t)\circ\sigma)(x).$$
This complete the proof of Corollary 1.2.
\end{proof}

\end{document}